\documentclass[11pt, a4paper,draft]{article}
\usepackage{a4wide}
\usepackage{mathrsfs} 
\usepackage{amssymb} 
\usepackage{amsmath} 
\usepackage{amsthm} 

\newtheorem{theorem}{Theorem}
\newtheorem{proposition}[theorem]{Proposition}
\newtheorem{lemma}[theorem]{Lemma}

\newcommand{\G}[2]{G_{#1,#2}}
\newcommand{\Exp}{\,\mathbb{E}}
\renewcommand{\Pr}{\,\mathbb{P}}
\newcommand{\eps}{\varepsilon}
\newcommand{\given}{ \; \big| \; }

\DeclareMathOperator{\Bin}{Bin}
\DeclareMathOperator{\avgdeg}{\overline{\deg}}
\DeclareMathOperator{\ch}{ch}

\title{The $t$-improper chromatic number of random graphs}

\author{Ross J. Kang
\thanks{School of Computer Science, McGill University, Montr\'eal, Qu\'ebec, H2A 2A7, Canada.  Part of this work was completed while this author was a doctoral student at the University of Oxford.  He was partially supported by NSERC (Canada) and the Commonwealth Scholarships Commission (UK).}
\and
Colin McDiarmid
\thanks{Department of Statistics, University of Oxford, 1 South Parks Road, Oxford OX1 3TG, United Kingdom.}
}

\begin{document}

\maketitle

\begin{abstract}
We consider the $t$-improper chromatic number of the Erd{\H o}s-R{\'e}nyi random graph $\G{n}{p}$.  The t-improper chromatic number $\chi^t(G)$ is the smallest number of colours needed in a colouring of the vertices in which each colour class induces a subgraph of maximum degree at most $t$.  If $t = 0$, then this is the usual notion of proper colouring.
When the edge probability $p$ is constant, we provide a detailed description of the asymptotic behaviour of $\chi^t(\G{n}{p})$ over the range of choices for the growth of $t = t(n)$.
\end{abstract}

\section{Introduction}\label{sec:intro}

We consider the $t$-improper chromatic number of the Erd\H os-R\'enyi random graph $\G{n}{p}$.  As usual, $\G{n}{p}$ denotes a random graph with vertex set $[n] = \{1, \ldots, n \}$ in which the edges are included independently at random with probability $p$. The {\it $t$-dependence number} $\alpha^t(G)$ of a graph $G$ is the maximum size of a {\it $t$-dependent set} --- a vertex subset which induces a subgraph of maximum degree at most $t$. The {\it $t$-improper chromatic number} $\chi^t(G)$ is the smallest number of colours needed in a {\it $t$-improper colouring} --- a colouring of the vertices in which colour classes are $t$-dependent sets.  Note that $\chi^t(G) \ge |V(G)|/\alpha^t(G)$ for any graph $G$ and any integer $t$.

The $t$-improper chromatic number was introduced about two decades ago independently by Andrews and Jacobson~\cite{AnJa85}, Harary and Fraughnaugh (n\'ee Jones)~\cite{Har85, HaJo85}, and Cowen {\it et al.}~\cite{CCW86}.  In the first paper, the authors considered various general lower bounds for the $t$-improper chromatic number; in the second, the authors studied $\chi^t$ as part of the larger setting of generalised chromatic numbers; 
in the third, the authors established best upper bounds on $\chi^t$ for planar graphs to generalise the Four Colour Theorem.  Several papers on the topic have since appeared; 
for instance, two papers, by Eaton and Hull~\cite{EaHu99} and \v{S}krekovski~\cite{Skr99}, extend the program of Cowen {\it et al.}~to a list colouring variant $\ch^t$ of $\chi^t$ and both pose the question: is $\ch^1(G) \le 4$ for every planar graph $G$?

Clearly, when $t = 0$, we are simply considering the ordinary notion of the chromatic number of random graphs, and this topic is well studied.  Fix $0 < p < 1$ and let $b = 1/(1 - p)$.  In 1975, Grimmett and McDiarmid~\cite{GrMc75} conjectured that $\chi(\G{n}{p}) \sim n/(2 \log_b n)$ asymptotically almost surely (a.a.s.).  This remained a major open problem in random graph theory for over a decade, until Bollob\'as~\cite{Bol88} and Matula and Ku\v cera~\cite{MaKu90} used martingale techniques to establish the conjecture.  {\L}uczak~\cite{Luc91a} extended the result to sparse random graphs.  For further background into the colouring of random graphs,
consult~\cite{JLR00, Bol01}.  The main objective of this paper is to extend this study to $t$-improper colouring.

Before we describe our main results, let us make some basic observations about the behaviour of the $t$-improper chromatic number.  Let $G$ be a graph and $t$ a non-negative integer.  Since a $t$-dependent set is $(t + 1)$-dependent, it follows that $\chi^t(G) \ge \chi^{t+1}(G)$.  Also, since each colour class of a $t$-improper colouring is $t$-dependent and so can be properly coloured with at most $t + 1$ colours, it follows that $\chi^t(G) \ge \chi(G)/(t+1)$.  Furthermore, it is straightforward to derive from a decomposition theorem of Lov\'asz~\cite{Lov66} (cf.~\cite{CGJ97}) that $\chi^t(G) \le \lceil (\Delta(G) + 1)/(t + 1) \rceil$ where $\Delta(G)$ denotes the maximum degree of $G$.  As a consequence of these last three observations, we obtain the following range of values for $\chi^t(G)$.
\begin{proposition}\label{prop:Lov}
For any graph $G$ and non-negative integer $t$,
\[ \frac{\chi(G)}{t + 1} \le \chi^t(G) \le \min\left\{ \left\lceil \frac{\Delta(G) + 1}{t + 1} \right\rceil, \chi(G) \right\}.\]
\end{proposition}

In this paper, our main focus is on dense random graphs --- i.e.~the case in which the edge probability $p$ is a fixed constant between $0$ and $1$.  Recall that $\Delta(\G{n}{p}) \sim n p$ a.a.s.~in this case. 
We show that $\chi^t(\G{n}{p})$ is likely to be close to the upper end of the range in Proposition~\ref{prop:Lov}, as long as $t(n) = o(\ln n)$ or $t(n) = \omega(\ln n)$.
We also give a precise description of the behaviour of $\chi^t(\G{n}{p})$ for the intermediate case $t(n) = \Theta(\ln n)$.  Here is our main theorem.

\begin{theorem}\label{thm:dense,new}
Fix $0 < p < 1$ and let $b = 1/(1-p)$.  There exists a function $\kappa_p = \kappa_p(\tau)$ that is continuous and strictly increasing for $\tau \in [0,\infty)$, with $\kappa_p(0) = 2/\ln b$ and $\kappa_p(\tau) \sim \tau/p$ as $\tau \to \infty$ such that the following holds: a.a.s.
\[
\chi^t(\G{n}{p}) \sim \frac{n}{\kappa_p(t/\ln n) \ln n}
\]
if $t(n) = o(n)$.  Furthermore, if $t(n) \sim n p/x$, where $x > 0$ is fixed and not integral, then $\chi^t(\G{n}{p}) = \lceil x \rceil$ a.a.s.
\end{theorem}

\noindent
In Sections~\ref{sec:largedeviations} and~\ref{sec:secmom}, we use large deviations techniques and a second moment calculation to develop a fairly precise description of the $t$-dependence number of $\G{n}{p}$.  By these computations, we identify the function $\kappa_p$.  A proof of the main theorem is given in Section~\ref{sec:proof}.

For sparse random graphs --- i.e.~when $p(n) = o(1)$ --- we will give just one result.  For further results on the $t$-improper chromatic number in this regime, see Section~4.3 of~\cite{Kan08}.

\begin{theorem}\label{thm:sparse}
Suppose $0 < p(n) < 1$, $p(n) = o(1)$ and $\eps > 0$.  Set $d(n) = n p(n)$.  There exist constants $d_0$ and $\tau > 0$ such that, if $d(n) \ge d_0$ and $t(n) \le \tau \ln d$, then $(1 - \eps) d/(2 \ln d) \le \chi^t(\G{n}{p}) \le (1 + \eps) d/(2 \ln d)$ a.a.s.
\end{theorem}

\noindent
The upper bound here follows immediately from the upper bound of {\L}uczak~\cite{Luc91a} on the chromatic number of sparse random graphs.  A corollary is that, if $d(n) \to \infty$ and $t(n) = o(\ln d)$ as $n \to \infty$, then $\chi^t(\G{n}{p}) \sim d/(2 \ln d)$ a.a.s.
We shall prove Theorem~\ref{thm:sparse} in Section~\ref{sec:sparse}.

For fixed $t$, the property of a set being $t$-dependent is an hereditary property so that the results of Scheinerman~\cite{Sch92} and Bollob\'as and Thomason~\cite{BoTh95} apply --- but in our work it is important that $t$ is allowed to vary.

\section{The expected number of $t$-dependent $k$-sets}\label{sec:largedeviations}

In this section, we use large deviations results to describe the behaviour of the expected number of $t$-dependent $k$-sets.  This estimation provides us with an immediate lower bound for $\chi^t(\G{n}{p})$.
For background into large deviations, consult~\cite{DeZe98}; we borrow some notation from this reference.  Given $0 < p < 1$, we let $q = 1 - p$ throughout.  Also, let
\[
\Lambda^*(x) = \left\{ \begin{array}{ll}
\displaystyle x \ln \frac{x}{p} + (1 - x) \ln \frac{1 - x}{q} & \mbox{for $x\in [0, 1]$}\\
\infty & \mbox{otherwise}
\end{array} \right.
\]
(where $\Lambda^*(0) = \ln (1/q)$ and $\Lambda^*(1) = \ln (1/p)$).
This is the Fenchel-Legendre transform of the logarithmic moment generating function associated with the Bernoulli distribution with probability $p$ (cf.~Exercise 2.2.23(b) of~\cite{DeZe98}).  Some easy calculus checks that $\Lambda^*(x)$ has a global minimum of $0$ at $x = p$, is strictly decreasing on $[0, p)$ and strictly increasing on $(p,1]$.  We note the following large deviations result for the binomial distribution.

\begin{lemma} \label{lem.bindev}
There is a constant $\delta>0$ such that the following holds.  Let $0<p<1$, let $n$ be a positive integer,
and let $X \in \Bin(n,p)$.  Then for each positive integer $k \le np$,
\[
\delta \cdot\max\left\{k^{-1/2},(n-k)^{-1/2}\right\} \cdot \exp(-n\Lambda^*( k/n ))
\le
\Pr( X \le k)
\le
\exp(-n\Lambda^*( k/n ))
.
\]
\end{lemma}

\noindent
Also, of course $\Pr( X=0) = q^n = \exp(-n\Lambda^*(0))$.
Furthermore, the monotonicity of $\Lambda^*(x)$ shows that the right inequality holds also for non-integral $k$.
For convenience, we give a proof of this lemma in the appendix to this paper.
For related very general results, see for example the monograph of Dembo and Zeitouni~\cite{DeZe98}. 
Lemma~\ref{lem.bindev} immediately yields the following estimate for the number of 
$t$-dependent $k$-sets.  For a graph $G$, we let $\avgdeg(G)$ denote the average degree of $G$.

\begin{lemma}\label{lem:A_n}
Suppose $0 < p = p(n) < 1$ and suppose the positive integers $t = t(n)$ and $k = k(n)$ satisfy that $t \le p(k - 1)$.
\begin{enumerate}
\item\label{lem:A_n,i} $\displaystyle \Pr(\avgdeg(\G{k}{p}) \le t) \le \exp\left(-\binom{k}{2} \Lambda^*\left( \frac{t}{k - 1} \right) \right)$; and
\item\label{lem:A_n,ii} $\displaystyle \Pr(\avgdeg(\G{k}{p}) \le t) \ge \exp\left(-\binom{k}{2} \Lambda^*\left( \frac{t}{k - 1} \right) - \ln k +O(1) \right)$.
\end{enumerate}
\end{lemma}

\noindent
Since a $t$-dependent $k$-set has average degree at most $t$, Lemma~\ref{lem:A_n}\ref{lem:A_n,i} implies an upper bound on the expected $t$-dependence number of $\G{n}{p}$.  In particular, it shows that, if
\begin{align*} 
\binom{n}{k} \exp\left(-\binom{k}{2} \Lambda^*\left( \frac{t}{k - 1} \right) \right) = o(1),
\end{align*}
then $\alpha^t(\G{n}{p}) \le k$ a.a.s.  We define the function $\kappa_p(\tau)$ of Theorem~\ref{thm:dense,new} based on the range of $k$ (given $t$) for which the above condition holds.  Note first the following lemma, the straightforward proof of which is omitted.

\begin{lemma}\label{lem:kappa_0}
Fix $0 < p < 1$.  For any $\tau \ge 0$, there is a unique $\kappa_p(\tau) > \tau/p$ such that
\[
\frac{\kappa}{2} \Lambda^*\left( \frac{\tau}{\kappa} \right)
\begin{cases}
< 1 & \text{ if }\tau/p < \kappa < \kappa_p(\tau) \\
= 1 & \text{ if }\kappa = \kappa_p(\tau) \\
> 1 & \text{ if }\kappa > \kappa_p(\tau)
\end{cases}
.
\]
The function $\kappa_p(\tau)$ for $\tau \in [0, \infty)$ is continuous and strictly increasing, with $\kappa_p(0) = 2/\ln (1/q)$ and $\kappa_p(\tau) \sim \tau/p$ as $\tau \to \infty$.
\end{lemma}

Suppose $t(n)/\ln n \to \tau$ for some $\tau \ge 0$.  The main result of this section is to show that, if $\tau/p < \kappa < \kappa_p(\tau)$, then the expected number of $t$-dependent $(\kappa \ln n)$-sets goes to infinity and, if $\kappa > \kappa_p(\tau)$, then it goes to zero.
More precisely, we have the following.

\begin{theorem}\label{thm:Expthresh}
Fix $0 < p < 1$.  Fix $\tau, \kappa \ge 0$ with $\kappa > \tau / p$ and suppose $t(n)/\ln n \to \tau$ as $n \to \infty$ and $k(n) \sim \kappa \ln n$.
Let $\mathcal{S}_{n,t,k}$ be the collection of $t$-dependent $k$-sets in $\G{n}{p}$.  Then
\begin{align*}
\Exp(|\mathcal{S}_{n,t,k}|) = \exp\left(k \ln n \left(1 - \frac{\kappa}{2} \Lambda^*\left( \frac{\tau}{\kappa} \right) + o(1) \right)\right).
\end{align*}
\end{theorem}

\noindent
Before continuing with the proof of this theorem, we mention that there is a similar statement for the expected number of $t$-improper $\lceil n/k \rceil$-colourings of $\G{n}{p}$.  That is, under the same conditions as Theorem~\ref{thm:Expthresh}, if $\mathcal{C}_{n,t,k}$ is the collection of $t$-improper $\lceil n/k \rceil$-colourings of $\G{n}{p}$, then
\begin{align*}
\Exp(|\mathcal{C}_{n,t,k}|) = \exp\left(n \ln n \left(1 - \frac{\kappa}{2} \Lambda^*\left( \frac{\tau}{\kappa} \right) + o(1) \right)\right).
\end{align*}
A proof of this statement can be found in Section~4.2 of~\cite{Kan08}.
Observe that this generalises a result of
Grimmett and McDiarmid~\cite{GrMc75} concerning the expected number of proper $j$-colourings of $\G{n}{p}$.

\begin{proof}[Proof of Theorem~\ref{thm:Expthresh}]
Clearly, since $\kappa > \tau / p$, it follows that $t(n) \le p (k(n) - 1)$ for large enough $n$.  Thus, since $\left(n/k\right)^k \le \binom{n}{k} \le \left(e n/k\right)^k$, it follows from Lemma~\ref{lem:A_n}\ref{lem:A_n,i} and the continuity of $\Lambda^*$ that
\[ \Exp(|\mathcal{S}_{n,t,k}|) \le \binom{n}{k}\exp\left(-\binom{k}{2} \Lambda^*\left( \frac{t}{k - 1} \right) \right) = \exp\left( k \ln n \left( 1 - \frac{\kappa}{2} \Lambda^*\left( \frac{\tau}{\kappa} \right) + o(1) \right) \right) \]
and so we just need to show the reverse inequality.

Our approach will be to bound the probability that a $k$-set is $t$-dependent with an appropriately chosen conditional probability.  First, we will give an estimate for the conditional probability
\[ P_{n, \eps} = \Pr\left( \Delta(\G{k}{p}) > t \given \avgdeg(\G{k}{p}) \le (1 - \eps)t \right) \]
for $0 < \eps < 1$.  Note that, if we condition on a fixed number $m \in \{0, \ldots, \binom{k}{2}\}$ of edges in $\G{k}{p}$, this is essentially the uniform random graph model $\G{k}{m}$ (where we choose among all $\binom{\binom{k}{2}}{m}$ possible subgraphs with $m$ edges).  Thus,
$\Pr\left( \Delta(\G{k}{p}) > t \given |E(\G{k}{p})| = m \right) = \Pr(\Delta(\G{k}{m}) > t)$.  Also, it is clear by a coupling argument that, if there are more edges, then it is more likely that the maximum degree will be higher, i.e.~$\Pr(\Delta(\G{k}{m-1}) > t) \le \Pr(\Delta(\G{k}{m}) > t)$.  Now let $\hat{m} = \left\lfloor (1 - \eps) k t/2 \right\rfloor$.  It follows that
\begin{align*}
P_{n, \eps} & = \Pr\left( \Delta(\G{k}{p}) > t \given |E(\G{k}{p})| \le \hat{m} \right) \\
                   & \le \Pr\left( \Delta(\G{k}{\hat{m}} ) > t \right) \\
                   & \le k \Pr\left( \deg(v) > t \text{ in }\G{k}{\hat{m}} \right).
\end{align*}
The degree of a vertex in $\G{k}{\hat{m}}$ has a hypergeometric distribution with parameters $\binom{k}{2}$, $k - 1$ and $\hat{m}$ with expected value $\lambda = (k - 1)\hat{m}/\binom{k}{2}\le (1 - \eps) t$; and thus, by a Chernoff-Hoeffding inequality (cf.~Theorem 2.10 and Inequality~(2.5) of~\cite{JLR00}),
\[
P_{n, \eps} \le k \exp\left( -\frac{\eps^2 t^2}{2 t(1-2\eps/3)} \right) \le k \exp\left( -\frac{\eps^2}{2}t \right).
\]
If we choose $\eps = \eps_n$ approaching zero slowly enough, say, $\eps_n = (\ln n)^{-1/3}$, then this conditional probability is $o(1)$.  Then, furthermore, using Lemma~\ref{lem:A_n}\ref{lem:A_n,ii},
\begin{align*}
\Pr\left( \Delta(\G{k}{p}) \le t \right)
 & \ge \Pr\left( \Delta(\G{k}{p}) \le t \given \avgdeg(\G{k}{p}) \le (1 - \eps_n) t \right)
 \cdot \Pr\left( \avgdeg(\G{k}{p}) \le (1 - \eps_n) t \right) \\
 & = (1 - P_{n,\eps_n}) \Pr\left( \avgdeg(\G{k}{p}) \le (1 - \eps_n) t\right) \\
 & \ge (1 - o(1)) \exp\left( -\binom{k}{2} \left(\Lambda^*\left( \frac{\tau}{\kappa} \right) + o(1)\right)\right)
 \text{[since }\eps_n \to 0\text{ slowly enough]} \\
 & = \exp\left( -\binom{k}{2} \left(\Lambda^*\left( \frac{\tau}{\kappa} \right) + o(1)\right)\right)
\end{align*}
and the expected number of $t$-dependent $k$-sets satisfies
\begin{align*}
\Exp(|\mathcal{S}_{n,t,k}|) & \ge \binom{n}{k} \exp\left( -\binom{k}{2} \left(\Lambda^*\left( \frac{\tau}{\kappa} \right) + o(1)\right)\right) \\
                  & = \exp\left(k \ln n \left(1 - \frac{\kappa}{2} \Lambda^*\left( \frac{\tau}{\kappa} \right) + o(1) \right)\right)
\end{align*}
as required.
\end{proof}

\section{A second moment calculation}\label{sec:secmom}

In this section, we perform a second moment calculation which yields both a lower bound on the $t$-dependence number, as well as an upper bound on the $t$-improper chromatic number, for the case when $t(n)$ is of order $\ln n$.
We remark that the following lemma was posed as a conjecture in an earlier version of this work~\cite{KaMc07}.

\begin{lemma}\label{lem:t=Thetalogn}
Fix $0 < p < 1$.  Suppose $t(n)/\ln n \to \tau$ as $n \to \infty$ for some fixed $\tau \ge 0$.  If $\kappa_p(\tau)$ is as in Lemma~\ref{lem:kappa_0}, then $\chi^t(\G{n}{p}) \sim n/(\kappa_p(\tau) \ln n)$ a.a.s.
\end{lemma}

\noindent
In a parallel work together with N. Fountoulakis~\cite{FKM08+}, we have obtained a very precise description of both the $t$-dependence and the $t$-improper chromatic numbers of $\G{n}{p}$ in the case when $t$ and $p$ are both fixed constants.  We shall borrow some techniques from Section~4 of that work to show the following lemma.

\begin{lemma}\label{lem:t=Thetalogn,alpha}
Fix $0 < p < 1$.  Suppose $t(n)/\ln n \to \tau$ as $n \to \infty$ for some fixed $\tau \ge 0$.  If $\kappa_p(\tau)$ is as in Lemma~\ref{lem:kappa_0} and $k / \ln n \to \kappa_p(\tau) - \eps$ as $n \to \infty$ for some fixed $\eps > 0$, then $\Pr(\alpha^t(\G{n}{p}) < k) \le \exp(-\Omega(n^2/(\ln n)^5))$ a.a.s.
\end{lemma}

\noindent
Establishing a similar lemma for the ordinary chromatic number was key in~\cite{Bol88} to pinning down the asymptotic behaviour of the chromatic number of $\G{n}{p}$.  Before proceeding with the proof of this lemma, let us see how it implies Lemma~\ref{lem:t=Thetalogn}.

\begin{proof}[Proof of Lemma~\ref{lem:t=Thetalogn}]
For the lower bound, let $\kappa > \kappa_p(\tau)$.  If we let $\delta = (\kappa/2) \Lambda^*(\tau/\kappa) - 1$, then $\delta>0$ by Lemma 6.  But now, setting $k=\lceil \kappa \ln n \rceil$, we have from Theorem 7 that
\[
\Pr\left(\chi^t(G_{n,p}) \leq \frac{n}{\kappa \ln n} \right)
\leq \Pr\left(\alpha^t(G_{n,p}) \geq k\right)
\leq \Exp\left(|\mathcal{S}_{n,t,k}|\right)
=  \exp(-(\delta+o(1)) k \ln n)
\]
(where $\mathcal{S}_{n,t,k}$ is the collection of $t$-dependent $k$-sets in $\G{n}{p}$); thus, $\chi^t(G_{n,p}) \geq n/(\kappa \ln n)$ a.a.s.

Now for the upper bound,
suppose $k = \lceil \kappa \ln n \rceil$, where $\kappa = \kappa_p(\tau) - \eps/2 > 0$ for some fixed $\eps > 0$.
Let $\mathcal{A}_n$ denote the set of graphs $G$ on $[n]$ such that $\alpha^t(G[S]) \ge k$ for all $S \subseteq [n]$ with $|S| \ge n/(\ln n)^2$.  Then, by Lemma~\ref{lem:t=Thetalogn,alpha},
\begin{align*}
\Pr\left(\G{n}{p} \notin \mathcal{A}_n\right)
& \le 2^n \Pr\left(\alpha^t\left(\G{\lceil n/(\ln n)^2 \rceil}{p}\right) < k\right) \le \exp\left(O(n)-\Omega\left(n^2/(\ln n)^9\right)\right) \to 0
\end{align*}
as $n \to \infty$.  Therefore, $\G{n}{p} \in \mathcal{A}_n$ a.a.s.

But for a graph $G$ in $\mathcal{A}_n$ the following procedure will yield a colouring as desired.  Let $S' = [n]$.  While $|S'| \ge n/(\ln n)^2$, form a colour class from an arbitrary $t$-dependent $k$-subset $T$ of $S'$ and let $S' = S'\setminus T$.  At the end of these iterations, $|S'| < n/(\ln n)^2$ and we may just assign each vertex of $S'$ to its own colour class.  The resulting partition is a $t$-improper colouring of $\G{n}{p}$ and
the total number of colours used is less than $n/((\kappa_p(\tau) - \eps/2)\ln n) + n/(\ln n)^2 \le n/((\kappa_p(\tau) - \eps)\ln n)$ for large enough $n$.
\end{proof}

\noindent
For the proof of Lemma~\ref{lem:t=Thetalogn,alpha}, it is convenient to introduce one lemma, which is proved in the appendix.

\begin{lemma}\label{lem:mixedbin}
  Let $n_1$ and $n_2$ be positive integers, let $0<p<1$, and let $X$ and $Y$ be independent random variables with $X \in \Bin(n_1,p)$ and $Y/2 \in \Bin(n_2,p)$.
  Note that $\Exp(X+Y)= (n_1+2 n_2)p$.
  Then for $0\leq x \leq p$
  \[
  \Pr(X+Y \leq (n_1+2 n_2)x) \leq \exp \left( - \frac12 (n_1+2 n_2) \Lambda^*(x)\right).
  \]
\end{lemma}

\noindent
For comparison, note that, if instead of $Y/2 \in \Bin(n_2,p)$ we had $Y\in \Bin(2 n_2,p)$, then $X + Y \in \Bin(n_1+2 n_2,p)$ and so $\Pr(X + Y \le (n_1+2 n_2)x) \le \exp(-(n_1+2 n_2)\Lambda^*(x))$ by Lemma~\ref{lem.bindev}.

\begin{proof}[Proof of Lemma~\ref{lem:t=Thetalogn,alpha}]
We may assume without loss of generality that $0 < \eps < \kappa_p(\tau) - \tau/p$.
Let $\mathcal{S}_{n,t,k}$ be the collection of $t$-dependent $k$-sets in $\G{n}{p}$.  Let $\kappa = \kappa_p(\tau) - \eps$.  Notice that
$1 - (\kappa/2) \Lambda^*\left( \tau/\kappa \right) = \delta$
for some fixed $\delta > 0$ by Lemma~\ref{lem:kappa_0}; thus, by Theorem~\ref{thm:Expthresh},
\begin{align}
\Exp(|\mathcal{S}_{n,t,k}|) = \exp((\delta \kappa + o(1)) (\ln n)^2). \label{eqn:exp}
\end{align}
We use Janson's Inequality (Theorem~2.18(ii) in~\cite{JLR00}):
\begin{align}
\Pr(\alpha^t(\G{n}{p}) < k) = \Pr(|\mathcal{S}_{n,t,k}| = 0) \le \exp \left( - \frac{(\Exp(|\mathcal{S}_{n,t,k}|))^2}{\Exp(|\mathcal{S}_{n,t,k}|) + \Delta }\right), \label{eqn:Janson}
\end{align}
where
\[
\Delta = \sum_{A, B \subseteq [n], 1 < |A \cap B| < k} \Pr(A, B \in \mathcal{S}_{n,t,k}).
\]

Let $p(k,\ell)$ be the probability that two $k$-subsets of $[n]$ that
overlap on exactly $\ell$ vertices are both in $\mathcal{S}_{n,t,k}$.
We write
\begin{align*} 
\Delta =
 \sum_{\ell=2}^{\lfloor \lambda \ln n \rfloor} \binom{n}{k} \binom{k}{\ell} \binom{n - k}{k-\ell} p(k,\ell)
 + \sum_{\ell=\lfloor \lambda \ln n \rfloor + 1}^{k-1} \binom{n}{k} \binom{k}{\ell} \binom{n - k}{k-\ell} p(k,\ell) 
:=  \Delta_1 + \Delta_2.
\end{align*}
for some fixed $\lambda$ to be specified later.

We first bound $\Delta_1$.  Let $A$ and $B$ be two $k$-subsets of $[n]$ that overlap on exactly $\ell$ vertices, i.e.~$|A\cap B| = \ell$.  Then $p(k,\ell) = \Pr(A,B \in \mathcal{S}_{n,t,k}) = \Pr\left( A \in \mathcal{S}_{n,t,k} \given B \in \mathcal{S}_{n,t,k} \right) \Pr(B \in \mathcal{S}_{n,t,k})$.
The property of having maximum degree at most $t$ is monotone decreasing; so if we condition on the 
set $E$ of edges induced by $A\cap B$, then the conditional probability that $A \in \mathcal{S}_{n,t,k}$ is maximized when $E = \emptyset$.  Thus,
\begin{align*}
\Pr\left( A \in \mathcal{S}_{n,t,k} \given B \in \mathcal{S}_{n,t,k} \right)
& \le \Pr\left( A \in \mathcal{S}_{n,t,k} \given E = \emptyset \right) \le \frac{\Pr(A \in \mathcal{S}_{n,t,k})}{\Pr(E = \emptyset)} = b^{\binom{\ell}{2}} \Pr(A \in \mathcal{S}_{n,t,k})
\end{align*}
(where $b = 1/(1-p)$) implying that $p(k,\ell) \le b^{\binom{\ell}{2}}(\Pr(A \in \mathcal{S}_{n,t,k}))^2$.  We have though that
\begin{align}
\binom{k}{\ell} \binom{n-k}{k-\ell} \le k^\ell \frac{k^\ell}{(n-k)^\ell} \binom{n}{k}; \label{eqn:binom}
\end{align}
therefore, it follows that
\begin{align*}
\Delta_1
& \le \left(\binom{n}{k}\Pr(A \in \mathcal{S}_{n,t,k})\right)^2 \
\sum_{\ell=2}^{\lfloor \lambda \ln n \rfloor} \left(\frac{k^2}{n-k}\right)^\ell b^{\binom{\ell}{2}}
 = (\Exp(|\mathcal{S}_{n,t,k}|))^2 \sum_{\ell=2}^{\lfloor \lambda \ln n \rfloor} \left(\frac{k^2}{n-k}\right)^\ell b^{\binom{\ell}{2}}.
\end{align*}
If we set $s_{\ell} = (k^2/(n-k))^\ell  b^{\binom{\ell}{2}}$, then
\(
s_{\ell + 1}/s_{\ell} = b^{\ell} k^2/(n-k).
\)
Thus, the sequence $\{s_{\ell}\}$ is
strictly decreasing for $\ell < \log_b (n-k) - 2 \log_b k$ and is strictly increasing for
$\ell > \log_b (n-k) - 2 \log_b k$. So
\[ \max \{ s_{\ell} : 2 \le \ell \le \lfloor \lambda \ln n \rfloor \} \le
\max \left\{ s_2 ,  s_{\lambda \ln n} \right\}. \]
We have that $s_2 = b k^4/(n-k)^2$ and
\begin{align*}
s_{\lambda \ln n} = \left( \frac{k^2}{\sqrt{b}(n-k)} \cdot b^{(\lambda/2)\ln n}\right)^{\lambda \ln n} = \left( \frac{k^2 n}{\sqrt{b}(n-k)} \cdot n^{(\lambda/2)\ln b - 1}\right)^{\lambda \ln n}.
\end{align*}
Since $k = n^{o(1)}$, it is clear that, if $\lambda < 2/\ln b$, then $s_{\lambda \ln n} = \exp(-\Omega((\ln n)^2))$ and, in particular, $s_{\lambda \ln n} = o(s_2)$.  Therefore, if $\lambda < 2/\ln b$, then
$\Delta_1 = O((\ln n)^5/n^2) (\Exp(|\mathcal{S}_{n,t,k}|))^2$.


  Next, we bound $\Delta_2$.  We have that
\begin{align*}
  \Pr\left(A \in {\mathcal{S}}_{n,t,k} \given B \in {\mathcal{S}}_{n,t,k} \right)
  & \le
  \Pr(\forall v \in A \setminus B, \deg_A(v) \le t)
   \le
  \Pr\left(\sum_{v \in A \setminus B} \deg_A(v) \le t(k-\ell)\right)\\
  & =
  \Pr\left(\Bin(\ell(k-\ell),p)+ 2\Bin\left(\binom{k-\ell}{2},p\right) \le t(k-\ell) \right).
\end{align*}
  Now we use Lemma~\ref{lem:mixedbin}, with $n_1=\ell(k-\ell)$, $n_2=\binom{k-\ell}{2}$,
  and $x=t/(k-1)$ which is less than $p$ for $n$ sufficiently large.
  Note that $n_1+2 n_2= (k-1)(k-\ell)$ and so $(n_1+2 n_2)x=t(k-\ell)$.
  Hence, by Lemma~\ref{lem:mixedbin}, the last quantity displayed above is at most
\[
  \exp \left(-\frac12 (k-1)(k-\ell) \Lambda^*\left(\frac{t}{k\!-\!1} \right)\right)
  =\exp\left( -\left(1-\frac{\ell}{k}\right)\binom{k}{2} \Lambda^*\left( \frac{t}{k\!-\!1} \right)\right).
\]
  Thus we have shown that
\[
  \Pr\left( A \in {\mathcal{S}}_{n,t,k} \given B \in {\mathcal{S}}_{n,t,k} \right) \leq
  \exp\left( -\left(1-\frac{\ell}{k}\right)\binom{k}{2} \Lambda^*\left( \frac{t}{k\!-\!1} \right)\right).
\]
  Also, observe that 
\[
  \exp\left(\frac{\ell}{k} \binom{k}{2} \Lambda^*\left( \frac{t}{k\!-\!1} \right) \right) = n^{(1-\delta + o(1))\ell}.
\]
  Therefore, substituting the last two results together with~\eqref{eqn:binom}
  into the expression for $\Delta_2$, we obtain
\begin{align*}
  \Delta_2
  & \le
  \binom{n}{k}^2 \Pr(B \in \mathcal{S}_{n,t,k})
  \exp\left(-\binom{k}{2} \Lambda^*\left( \frac{t}{k - 1} \right) \right) \sum_{\ell=\lfloor \lambda \ln n \rfloor + 1}^{k-1}
  \left(\frac{k^2  n^{1-\delta + o(1)}}{n-k}\right)^\ell \\
  & =
  (\Exp(|\mathcal{S}_{n,t,k}|))^2 n^{o(\ln n)}
  \sum_{\ell=\lfloor \lambda \ln n \rfloor + 1}^{k-1}
  \left(n^{-\delta + o(1)}\right)^\ell \text{ [by Theorem~\ref{thm:Expthresh} and $k = O(\ln n)$]}\\
  & \le
  (\Exp(|\mathcal{S}_{n,t,k}|))^2 \exp\left(-(\delta \lambda + o(1))(\ln n)^2\right) = \; o\left(1/n^2\right)(\Exp(|\mathcal{S}_{n,t,k}|))^2
\end{align*}
  for any choice of $\lambda > 0$ fixed.

If we let $\lambda = 1/\ln b$, then
$\Delta = \Delta_1 + \Delta_2 = O\left((\ln n)^5/n^2\right) (\Exp(|\mathcal{S}_{n,t,k}|))^2$ and it now follows from~\eqref{eqn:exp} and~\eqref{eqn:Janson} that
$\Pr\left(\alpha^t(\G{n}{p}\right) < k) \le \exp\left(-\Omega\left(n^2/(\ln n)^5\right)\right)$.
\end{proof}


\section{Proof of Theorem~\ref{thm:dense,new}}\label{sec:proof}

Let $\eps > 0$.  Note that, by the upper bound of Proposition~\ref{prop:Lov} and the fact that $\Delta(\G{n}{p}) \sim n p$ a.a.s, we have that $\chi^t(\G{n}{p}) \le (1 + \eps) n p/t$ a.a.s.  This observation trivially gives us the required upper bound of Theorem~\ref{thm:dense,new} if $t(n)/\ln n \to \infty$ as $n \to \infty$ since then $\kappa_p(t/\ln n) \ln n \sim t/p$.  The lower bound in this case is implied by the following first moment calculation.

\begin{lemma}\label{lem:t=omegalogn}
Fix $0 < p < 1$ and $\eps > 0$.  There exists fixed $B = B(p,\eps)$ such that, if $t(n) \ge B \ln n$, then $\chi^t(\G{n}{p}) \ge (1-\eps) n p/t$ a.a.s.
\end{lemma}

\begin{proof}
Let $k = k(n) = \left\lceil t/((1 - \eps/2) p) \right\rceil + 1$ so that $k < t/((1 - \eps) p)$ for large enough $n$.  Thus, by Lemma~\ref{lem:A_n}\ref{lem:A_n,i}, if $\mathcal{S}_{n,t,k}$ is the collection of $t$-dependent $k$-sets in $\G{n}{p}$, then
\begin{align*}
\Exp(|\mathcal{S}_{n,t,k}|)
& \le \binom{n}{k}\exp\left(-\binom{k}{2} \Lambda^*\left( \frac{t}{k - 1} \right) \right)
 \le n^k \exp\left(-\frac{k^2}3 \Lambda^*\left( (1-\eps/2)p \right) \right)
 = \left(n e^{-C t}\right)^k
\end{align*}
for some fixed $C = C(p,\eps) > 0$.  Letting $B = 2/C$, we obtain that $\Exp(\mathcal{S}_{n,t,k}) \le n^{-k} \to 0$
as $n\to \infty$.  So, with probability going to one,
$\alpha^t(\G{n}{p}) \le t/((1-\eps)p)$ and $\chi^t(\G{n}{p}) \ge (1-\eps)n p/t$.
\end{proof}

\noindent
Therefore, to prove Theorem~\ref{thm:dense,new}, it suffices to consider $t = t(n) < B \ln n$ where $B = B(p,\eps)$ is as in the above lemma.  We may assume that $\eps < 1$.  Recall that $\kappa_p(\tau)$ is continuous on the compact set $[0,B]$ and $\kappa_p(\tau) \ge \kappa_p(0) > 0$ for any $\tau \ge 0$.  Hence, there exist $0 = \tau_0 < \tau_1 < \cdots < \tau_j = B$ (for some $j$) such that $\kappa_p(\tau_{i+1}) \le (1 + \eps/3)\kappa_p(\tau_i)$ for each $i \in \{0, \ldots, j-1\}$.  Since there are only finitely many intervals $[\tau_i, \tau_{i+1}]$, it suffices to consider one, and assume that $\tau_i \ln n \le t(n) \le \tau_{i+1} \ln n$ for all $n$ large enough.
By Lemma~\ref{lem:t=Thetalogn} and the monotonicity of $\chi^t$ and $\kappa_p(\tau)$,
\begin{align*}
\chi^t(\G{n}{p})
& \le \chi^{\tau_i \ln n}(\G{n}{p}) \le \left(1 + \frac{\eps}{3}\right) \frac{n}{\kappa_p(\tau_i) \ln n}\\
& \le \left(1 + \frac{\eps}{3}\right) \left(1 + \frac{\eps}{3}\right) \frac{n}{\kappa_p(\tau_{i+1}) \ln n}
 \le (1 + \eps) \frac{n}{\kappa_p(t/\ln n) \ln n}
\end{align*}
a.a.s. Similarly,
\begin{align*}
\chi^t(\G{n}{p})
& \ge \chi^{\tau_{i+1} \ln n}(\G{n}{p}) \ge \left(1 - \frac{\eps}{3}\right) \frac{n}{\kappa_p(\tau_{i+1}) \ln n}\\
& \ge \left(1 - \frac{\eps}{3}\right) \left(1 + \frac{\eps}{3}\right)^{-1} \frac{n}{\kappa_p(\tau_i) \ln n}
 \ge (1 - \eps) \frac{n}{\kappa_p(t/\ln n) \ln n}
\end{align*}
a.a.s.  This completes the proof.


\section{Sparse graphs}\label{sec:sparse}

In previous sections, we considered dense random graphs, for which the expected average degree is $\Theta(n)$.  In this section, we consider random graphs with smaller expected average degree, i.e.~$d(n) = n p(n) = o(n)$.  For Theorem~\ref{thm:sparse}, the upper bound follows from the result of {\L}uczak~\cite{Luc91a} for the chromatic number, as we already noted.  For the lower bound, let us prove more by using the methods of Section~\ref{sec:largedeviations}.
First note the following lemma, the elementary proof of which is omitted.

\begin{lemma}\label{lem:kappa_0,sparse}
For any $\tau \ge 0$, there is a unique $\kappa(\tau) > \tau$ such that
\[
\frac{1}{2}\left( \kappa - \tau - \tau \ln \frac{\kappa}{\tau} \right)
\begin{cases}
< 1 & \text{ if }\tau < \kappa < \kappa(\tau) \\
= 1 & \text{ if }\kappa = \kappa(\tau) \\
> 1 & \text{ if }\kappa > \kappa(\tau)
\end{cases}
.
\]
The function $\kappa(\tau)$ for $\tau \in [0, \infty)$ is continuous and strictly increasing, with $\kappa(0) = 2$ and $\kappa(\tau) \sim \tau$ as $\tau \to \infty$.
\end{lemma}

\begin{lemma}\label{lem:sparse,t=Thetalogd}
Suppose $0 < p(n) < 1$, $p(n) = o(1)$ and let $d(n) = n p(n) \ge 4$ for $n$ sufficiently large.  Suppose that $t(n) /\ln d(n) \to \tau$ as $n\to \infty$ for some constant $\tau \ge 0$ and let $\kappa$ be a constant such that $\kappa > \kappa(\tau)$ where $\kappa(\tau)$ is as defined in Lemma~\ref{lem:kappa_0,sparse}.
Then $\chi^t(\G{n}{p}) \ge d/(\kappa \ln d)$ a.a.s.
\end{lemma}

\noindent
We first show how Theorem~\ref{thm:sparse} follows easily from the above lemma.

\begin{proof}[Proof of Theorem~\ref{thm:sparse}]
As noted above, we need only prove the lower bound on $\chi^t(\G{n}{p})$.  Since $\kappa(0) = 2$ and $\kappa(\cdot)$ is continuous and strictly increasing, we may choose $\tau > 0$ so that $2 < \kappa(\tau) < 2/(1 - \eps)$.  By Lemma~\ref{lem:sparse,t=Thetalogd}, there exists a constant $d_0$ such that, if $d(n) \ge d_0$, then $\chi^{\tau \ln d}(\G{n}{p}) \ge d/((2/(1 - \eps))\ln d)$.  The theorem now follows from the monotonicity of $\chi^t$.
\end{proof}

\begin{proof}[Proof of Lemma~\ref{lem:sparse,t=Thetalogd}]
Let $k = k(n) = \left\lceil \kappa\ln d(n)/p \right\rceil + 1$ where $\kappa > \kappa(\tau)$.  Since $\kappa > \tau$, it follows that $t(n) \le p(n)(k(n) - 1)$ for large enough $n$.  By the Chernoff Inequality~(2.6) of~\cite{JLR00},
\begin{align*}
\Pr(\avgdeg(\G{k}{p}) \le t)
& \le \exp\left( -\binom{k}{2}p \left( \frac{t}{p(k-1)} \ln \frac{t}{p(k-1)} + 1 - \frac{t}{p(k-1)} \right) \right) \\
& = \exp\left( - \left( \frac{\kappa}{2} \left( \frac{\tau}{\kappa} \log \frac{\tau}{\kappa} + 1 - \frac{\tau}{\kappa} \right) + o(1) \right) k \ln d  \right).
\end{align*}
(To apply the inequality from~\cite{JLR00}, we put $\lambda = \binom{k}{2}p$ and $t/\lambda = 1 - t(n)/(p(k-1))$, where $t(n)$ is `our' $t$.)
Let $\delta = (\kappa - \tau - \tau \ln (\kappa/\tau))/2 -1$, so that  $\delta>0$ by Lemma~\ref{lem:kappa_0,sparse}.  We have just seen that
$\Pr\left(\avgdeg(\G{k}{p}) \le t\right) \le \exp (-(\delta +1 +o(1)) k \ln d)$.
Now, $e n/(k d) < e/(\kappa \ln d)<1$ since $\kappa \geq 2$ and $d \geq 4$; and so $\ln {n \choose k} \leq k \ln (e n/k) \leq k \ln d$.
It follows that, if $\mathcal{S}_{n,t,k}$ is the collection of $t$-dependent $k$-sets in $\G{n}{p}$, then
$\ln \Exp(|\mathcal{S}_{n,t,k}|)
\le  \left( - \delta + o(1) \right) k \ln d
\to - \infty
$
as $n \to \infty$ and this completes the proof.
\end{proof}

\section{Concluding remarks}

In this paper, we have provided a detailed description of the behaviour of the $t$-improper chromatic number of dense random graphs $\G{n}{p}$ over the range of choices for the growth of $t = t(n)$.  We also briefly considered the $t$-improper chromatic number of sparse random graphs.
In this setting, perhaps the lower bound implied by Lemma~\ref{lem:sparse,t=Thetalogd} is tight?  This would correspond to a sparse analogue of Lemma~\ref{lem:t=Thetalogn}.  In Section~4.3 of~\cite{Kan08}, we showed that a sparse analogue of Theorem~\ref{thm:Expthresh} holds as long as $p = n^{-o(1)}$ (cf.~Theorem~4.17).  It is straightforward to adapt the proof of Lemma~\ref{lem:t=Thetalogn} to this range of $p$, but it is an open problem to extend the result to smaller choices of $p$.

Our results contrast with the behaviour of random geometric graphs, where $\chi^t$ is likely to be close to $\chi/(t + 1)$ for $t$ smaller than the expected average degree -- see~\cite{KMS08}.

\subsection*{Acknowledgement}

We would like to thank the referee for a careful reading and helpful comments.

\bibliographystyle{abbrv}
\bibliography{randim}

\appendix

\section{Proofs of two large deviations tools}

\begin{proof}[Proof of Lemma~\ref{lem.bindev}]
We first prove the right inequality.
The following basic inequality is well known (cf.~(2.4) of~\cite{JLR00} or Lemma~2.2 of~\cite{McD98}).
Let $X \in \Bin(n,p)$, where $0<p<1$.  Then for $0 \le t < q$
\begin{align*}
\Pr(X \ge n p + n t) \le \left( \left(\frac{p}{p+t}\right)^{p+t} \left(\frac{q}{q-t}\right)^{q-t} \right)^n.
\end{align*}
It follows that for $0 \le t < p$
\begin{align*}
\Pr(X \le n p - n t)
 & = \Pr( n - X \ge n q + n t) = \Pr( Y \ge n q + n t)\\
 & \le \left( \left(\frac{q}{q+t}\right)^{q+t} \left(\frac{p}{p-t}\right)^{p-t} \right)^n = \exp( -n\Lambda^*(p-t) ),
\end{align*}
where $X \in \Bin(n,p)$ and $Y \in \Bin(n,q)$.
Thus, putting $t = p - k/n$ we obtain
\begin{align*}
\Pr(X \le k)
 & = \Pr(X \le n p - n t)
 \le \exp( -n\Lambda^*(k/n) ).
\end{align*}

 For the left inequality, we use the sharp form of Stirling's formula due to Robbins, see for example inequality (1.4) of~\cite{Bol01}. Note that $1/(12 k) + 1/(12(n-k)) = n/(12 k(n-k))$; thus, for $1 \le k \le n-1$,
\begin{align*}
  \binom{n}{k} p^k q^{n - k}
& \ge
  \left( \frac{n p}{k} \right)^{k}
  \left( \frac{n q}{n - k} \right)^{n - k}
  \left(\frac{n}{2 \pi k(n - k)}\right)^{1/2}
  \exp\left(- \frac{n}{12 k(n-k)} \right)\\
& \ge
  \exp(-n\Lambda^*( k/n )) \
  (2 \pi)^{-1/2} \max\left\{k^{-1/2},(n-k)^{-1/2}\right\} \
  e^{-1/6}\\
 & \ge
 \delta \cdot \max\left\{k^{-1/2},(n-k)^{-1/2}\right\} \cdot \exp(-n\Lambda^*( k/n )) 
\end{align*}
where $\delta = (2 \pi)^{-1/2} e^{-1/6}$.
\end{proof}

\begin{proof}[Proof of Lemma~\ref{lem:mixedbin}]
  Let $\bar{X}=n_1-X$ so $\bar{X} \in \Bin(n_1,q)$, and let
  $\bar{Y} = 2 n_2-Y$ so $\bar{Y}/2 \in \Bin(n_2,q)$. Then
  $\Pr\left(X+Y \leq (n_1+2 n_2)x\right) = \Pr\left(\bar{X}+\bar{Y} \geq (n_1+2 n_2)(1-x)\right)$.
  Now for any real $u$ we have
\[(p+q e^u)^2 = p + q e^{2 u} - pq(1-e^u)^2 \le p+ q e^{2u},\]
  and so
\[\Exp\left(e^{u\bar{X}}\right) = (p+qe^u)^{n_1} \leq \left(p+qe^{2u}\right)^{n_1/2}.\]
  Thus,
  $\Exp\left(e^{u(\bar{X}+\bar{Y})}\right) \leq \left(p+qe^{2u}\right)^{(n_1+2 n_2)/2}$.
  Hence, for any $u \geq 0$, by Markov's inequality,
\begin{align*}
  \Pr\left(\bar{X}+\bar{Y} \geq (n_1+2 n_2)(1-x)\right)
  & \leq e^{-u(n_1+2 n_2)(1-x)} \left(p+qe^{2u}\right)^{(n_1+2 n_2)/2} \\
  & = \left(e^{-2u (1-x)} \left(p+qe^{2u}\right)\right)^{(n_1+2 n_2)/2}.
\end{align*}
  Note that $p(1-x)/(q x) \geq 1$.
  Now choose $u \geq 0$ such that $e^{2u} = p(1-x)/(qx)$ to obtain the desired inequality.
\end{proof}

\end{document}